\documentclass[]{interact}
\usepackage{latexsym}
\usepackage{anyfontsize} 
\usepackage{tikz} 
\usetikzlibrary{cd} 
\usepackage{color}
\usepackage{hyperref}
\usepackage{url}
\usepackage{breakurl}
\newcommand{\bburl}[1]{\textcolor{blue}{\url{#1}}}

\makeatletter
\newcommand{\monthyear}[1]{%
  \def\@monthyear{\uppercase{#1}}}
\newcommand{\volnumber}[1]{%
  \def\@volnumber{\uppercase{#1}}}
\makeatother

\theoremstyle{plain}
\numberwithin{equation}{section} 
\newtheorem{thm}{Theorem}[section] 
\newtheorem{theorem}[thm]{Theorem}
\newtheorem{lemma}[thm]{Lemma}
\newtheorem{example}[thm]{Example}
\newtheorem{definition}[thm]{Definition}
\newtheorem{proposition}[thm]{Proposition}
\newtheorem{corollary}[thm]{Corollary}

\numberwithin{table}{section} 
\numberwithin{figure}{section}

\def\NN{\mathbb{N}}
\def\PP{\mathbb{P}}
\def\ZZ{\mathbb{Z}}
\def\FF{\mathbb{F}}

\def\gooddoldprimes{\mathbb{G}_{\rm{D}}}
\def\baddoldprimes{\mathbb{B}_{\rm{D}}}
\def\goodrealizableprimes{\mathbb{G}_{\rm{R}}}
\def\badrealizableprimes{\mathbb{B}_{\rm{R}}}
\newcommand\ord{{\rm{ord}}}
\newcommand\failure{{\rm{Fail}}}
\newcommand\fix{{\rm{Fix}}}
\newcommand\least{{\mathcal{L}}}

\newcommand\trace{{\rm{trace}}}
\newcommand\notdivides{\mathrel{\kern-3pt\not\!\kern4.3pt\bigm|}}
\newcommand\divides\mid
\newcommand\smalldivides{\mathrel{\kern-2pt\kern3.5pt|}}
\newcommand\smallnotdivides{\mathrel{\kern-2pt\not\!\kern3.5pt|}}
\def\ge{\geqslant}
\def\le{\leqslant}

\begin{document}

\monthyear{Month Year}
\volnumber{Volume, Number}
\setcounter{page}{1}

\title{The local Dold congruence for Fibonacci and Lucas sequences}

\author{
\name{Sompong Chuysurichay,\textsuperscript{a}
Sawian Jaidee,\textsuperscript{b}
Chatchawan Panraksa,\textsuperscript{c}
and
Thomas Ward\textsuperscript{d}}
\affil{\textsuperscript{a}
Prince of Songkla University,
15 Karnjanavanich Road, Hat Yai,
Songkhla, 90110, Thailand
\textsuperscript{b}
Department of Mathematics,
Khon Kaen University,
Khon Kaen,
40002,
Muang District,
Thailand
\textsuperscript{c}
Mahidol University International College,
999 Phutthamonthon 4 Rd.,
Salaya, Phutthamonthon,
Nakhon Pathom,
73170,
Thailand
\textsuperscript{d}
Department of Mathematical Sciences,
Durham University,
Durham,
DH1 3LE,
County Durham,
United Kingdom}
}

\maketitle

{\bf Article type}: research 
\bigskip

\begin{abstract}
We discuss local properties of the classical
Fibonacci and Lucas sequences, characterising
in both cases those primes~$p$ for which the~$p$-part
satisfies the Dold congruences or is realizable. We calculate
an exact `repair' multiplier in each case and determine the
relative prime densities of the Dold, realizable, almost Dold,
and almost realizable prime sets. In that order, the four densities
are~$(\frac34,\frac13,1,\frac13)$ for the Lucas sequence
and~$(\frac12,\frac12,1,1)$ for the Fibonacci sequence.
\end{abstract}

\begin{keywords}
Arithmetic properties of recurrence sequences; realizability;
local Dold congruence; Lucas sequence; Fibonacci sequence
\end{keywords}

\section{Introduction and Definitions}

Recently Jaidee {\it{et al.}}~\cite{MR5076929} studied local
properties of integer sequences from the point of view
of realizability, a concept from a combinatorial
viewpoint on dynamical systems.
Our first interest here is to study
arithmetic aspects of the questions arising
in this context with an emphasis on linear recurrence sequences and
to explain some of the numerical observations made there.

We write~$\PP$ for the set of rational primes,~$\NN=\{1,2,\dots\}$
for the natural numbers,~$\ZZ$ for the integers,
and~$\NN_0$ for the non-negative integers.
For~$p\in\PP$ the~$p$-part of an integer~$n$
is written~$[n]_p=|n|_p^{-1}$. For an integer sequence~$a=(a_n)_{n\in\NN}$
we write
\[
[a]_p=([a_n]_p)_{n\in\NN}
\]
for the
sequence of~$p$-parts, and refer to a property satisfied by~$[a]_p$
as being true of the sequence~$a$ \emph{locally at~$p$}.
An integer sequence~$a$ satisfies the \emph{Dold congruence}
(or is Dold) if
\begin{equation*}\label{equationDoldForIntegerSequence}
\least_n(a)
=
\sum_{d\smalldivides n}\mu\bigl(\frac{n}{d}\bigr)a_d
\equiv
0\pmod{n}
\end{equation*}
for all~$n\in\NN$
and is said to be \emph{realizable} if
in addition~$\least_n(a)\ge0$ for all~$n\in\NN$.
Realizable means that the
sequence counts periodic points for some map,
in which case
the map is said to `realize' the sequence.
Notice that realizability in particular implies that the
original
sequence itself consists of non-negative integers.
The Dold congruence has a different formulation
due to Arias de Reyna~\cite{zbMATH02203410} as follows.
An integer sequence~$a$ is Dold if and only
if and only if
\begin{equation}\label{equationDoldForIntegerSequence2}
a_{mq^s}
\equiv
a_{mq^{s-1}}
\pmod{q^s}
\end{equation}
for all~$q\in\PP$,~$m\in\NN\setminus q\mathbb{N}$, and~$s\ge1$.
We refer to the survey by Byszewski {\it{et al.}}~\cite{MR4332826}
for an explanation of this terminology
and to Minton~\cite{MR3195758} for a classification of
these properties for linear recurrence sequences.

An integer sequence may fail to be realizable for
many different reasons, but Miska and Ward~\cite{MR4361581}
explored
a particularly simple notion of \emph{almost realizability},
where the failure can be repaired easily,
which in our context also applies to the Dold property.

\begin{definition}[Almost Dold sequences]\label{definitionAlmostDold}
For an integer sequence~$a$,
write~$\failure(a)$
for the least common multiple of
the denominators of the sequence~$(\frac{1}{n}\least_n(a))_{n\in\NN}$
if this is finite, in which case~$a$
is called \emph{almost Dold}.
An almost Dold sequence with~$\least_n(a)\ge0$
for all~$n\in\NN$ is called \emph{almost realizable}.
\end{definition}

This property leads to the notion of `repairing' a sequence
in the following sense. If~$\failure(a)$ exists (that is, is finite)
then multiplying~$a$ by~$\failure(a)$ produces a Dold sequence
and~$\failure(a)$ is the smallest integer with this property.
We will see later that the Lucas sequence gives rise
to an infinite family of non-trivial almost Dold sequences
to sit alongside the Stirling numbers identified in~\cite{MR4361581}.

\begin{definition}[Good primes]
Given an integer sequence~$a$ define two
associated sets of
primes as follows:
\begin{align*}
\gooddoldprimes(a)
&=
\{p\in\PP :
a\mbox{ is Dold at }p\},\\
\goodrealizableprimes(a)
&=
\{p\in\PP :
a\mbox{ is realizable at }p\}
\end{align*}
and similarly define~$\gooddoldprimes^{\text{almost}}$ and~$\goodrealizableprimes^{\text{almost}}$
for the analogous sets using the almost Dold and almost realizable
properties.
\end{definition}

We will also write~$\baddoldprimes(a)=\PP\setminus\gooddoldprimes(a)$
and~$\badrealizableprimes(a)=\PP\setminus\goodrealizableprimes(a)$
for the bad primes.
In the paper~\cite{MR5076929} these sets were discussed
for the Bernoulli denominators, the Bernoulli
numerators, and the Euler numbers in the particular
context of understanding realizability by a group
automorphism rather than an arbitrary map.
In those settings~$\gooddoldprimes(a)$ can be
viewed as a generalization of the notion of
regular prime (for the Bernoulli denominators)
for other sequences.

Clearly there are degenerate sequences for which all
primes are \emph{good},~$(1,1,1,\dots)$ being an example,
but the same phenomenon can arise in non-trivial ways.

\begin{example}[All primes can be good]
In~{\rm{\cite[Th.~4]{MR5076929}}} it is shown that an automorphism of
a nipotent group is locally nilpotently realizable at every
prime. That is, if~$\theta\colon G\to G$ is an automorphism
of a nilpotent group with the property that the
sequence~$a=(\fix_n(\theta))_{n\in\NN}$ defined by
\[
\fix_n(\theta)
=
\vert\{g\in G\colon\theta^ng=g\}\vert
\]
is finite for all~$n\in\NN$, then for any~$p\in\PP$
there is an automorphism~$\theta_p$ of a locally
nilpotent group~$G_p$ for which
\[
\fix_n(\theta_p)=[\fix_n(\theta)]_p
\]
for all~$n\in\NN$. That is,~$\goodrealizableprimes(a)=\PP$.
\end{example}

It is equally clear that there are degenerate sequences
for which all primes are \emph{bad}.
If~$a=(n)_{n\in\NN}$, then for any~$p\in\PP$ we have~$a_p=p$
and so~$a_p-a_1$ is not congruent to~$0$ modulo~$p$.
It follows that~$\baddoldprimes(a)=\PP$.

Our purpose here is to discuss these sets of
primes and their `almost' analogues for certain linear recurrence sequences.
We deliberately isolate some of the arguments for the Dold
property on its own, as this is the natural property for
arithmetic. Realizability is natural for dynamical systems,
but involves sign questions that are not natural in
arithmetic.
The picture that emerges is far from complete---indeed we
have nothing to say beyond a certain family of
binary recurrences---but is we hope of some interest.

\section{Lucas and Fibonacci numbers}

We will consider the Lucas sequence~$L=(1,3,4,\dots)$
and its companion the Fibonacci sequence~$F=(1,1,2,\dots)$
throughout this section. Define the two ranks of apparition
of a prime~$p$ by
\begin{align*}
\lambda(p)
&=
\begin{cases}
\min\{n\in\NN : p\divides L_n\}&\mbox{if $p$ divides some term of $L$},\\
\infty&\mbox{if not}
\end{cases}
\intertext{and}
\varphi(p)
&=
\begin{cases}
\min\{n\in\NN : p\divides F_n\}&\mbox{if $p$ divides some term of $F$},\\
\infty&\mbox{if not.}
\end{cases}
\end{align*}
It is well-known that
\begin{equation*}\label{rank1ForFibonacci}
\varphi(p)
\divides
p-\Bigl(\frac{5}{p}\Bigr)
\end{equation*}
for any odd prime~$p\neq5$, where
we write~$\bigl(\frac{5}{p}\bigr)$ for the Legendre symbol.

Clearly if~$p$ is a prime that does not divide
any term of~$L$ then~$p\in\goodrealizableprimes(L)$, so
we first examine those. Throughout we will make use of
well-known congruences and rules for rank of apparition
for the Lucas and Fibonacci sequences as they may be found, for example,
in Ribenboim's survey~\cite{MR1352481}.

\begin{lemma}\label{lemmaTrivialLucas}
Let~$p$ be a prime. Then the values of~$\lambda$ are
as follows.
\begin{enumerate}
\item[\rm(1)] $\lambda(2)=3$.
\item[\rm(2)] $\lambda(5)=\infty$.
\item[\rm(3)] If~$p\ne5$, then~$\lambda(p)<\infty$ if and only
if~$\varphi(p)$ is even. More precisely,
\begin{equation*}\label{eq:twolambda}
 \varphi(p)=2\lambda(p).
 \end{equation*}
\item[\rm(4)] If~$p\ne5$ and~$\left(\frac5p\right)=-1$, then
we have
\begin{equation}\label{oddrank}
\varphi(p)\text{ is odd}\quad\Longleftrightarrow\quad p\equiv1\pmod4.
\end{equation}
Consequently, if~$p\neq5$ and~$\left(\frac5p\right)=-1$ then
\[
\lambda(p)=\infty
\Longleftrightarrow
p\equiv13,17\pmod{20}
\]
and
\[
\lambda(p)
<
\infty
\Longrightarrow
p\equiv3,7\pmod{20}.
\]
\end{enumerate}
\end{lemma}

\begin{proof}
Firstly notice that~$2\divides L_3$,
proving~(1).
Now~$L_n=\frac{F_{2n}}{F_{n}}$
and~$[F_{2n}]_5=[F_n]_{5}$ so~$5$
does not divide any term of~$L$,
giving~(2).
If~$p\neq2,5$ then~$p\divides L_n$ if and only
if~$p\divides F_{2n}$ and~$p\notdivides F_{n}$.
It follows that~$p$ divides some term of~$L$
if and only if~$\varphi(p)$ is even.
Since~$p\divides F_m$ if and only if~$\varphi(p)\divides m$,
this becomes $\varphi(p)\divides 2n$ and $\varphi (p)\notdivides n$.
Such an index~$n$ exists precisely when~$\varphi(p)$ is even,
and the least such~$n$ is~$\frac{\varphi(p)}{2}$,
which shows~(3).
For~(4), assume that~$p\ne2,5$ and~$\left(\frac5p\right)=-1$
so that~$p\equiv\pm 2$ modulo~$5$
and~$5$ is a quadratic non-residue modulo~$p$ and
hence~$\FF_p[\sqrt{5}]=\FF_{p^2}$.
Let~$\alpha=\frac{(1+\sqrt5}{2}$,~$\beta=\frac{(1-\sqrt5)}{2}$,
and~$x=\frac{\alpha}{\beta}\in\mathbb F_{p^2}$.
The Frobenius exchanges~$\alpha$ and $\beta$,
so $x^p=x^{-1}$,~$\varphi(p)\divides p+1$,
and~$x=-\alpha^2$ so
\[
 x^{(p+1)/2}=(-1)^{(p+1)/2}\alpha^{p+1}=(-1)^{(p+3)/2}.
\]
If~$p\equiv1\pmod4$, then the right-hand side is~$1$
and $\frac{p+1}{2}$ is odd, so~$\varphi(p)$ is odd.
If~$p\equiv3\pmod4$, then the value~is~$-1$, so the order of~$x$ is even
and~$\varphi(p)$ is even and~\eqref{oddrank}
follows.
The residue-class consequences now follow by the Chinese remainder theorem.
\end{proof}

\begin{lemma}\label{lemmaMainLucas}
For~$p\in\PP$ with~$p\neq2,3$ we have~$p\in\gooddoldprimes(L)$ if
and only if~$\lambda(p)=\infty$ or~$2\lambda(p)\divides p-1$.
The primes~$2$ and~$3$ are not in~$\gooddoldprimes(L)$.
\end{lemma}

\begin{proof}
Write~$a_n=[L_n]_p$ for all~$n\ge1$ and~$p\in\mathbb{P}$
and~$r=\lambda(p)$ for~$p\in\mathbb{P}$ throughout,
with the prime~$p$ being clear from context as we go through.

For~$p=2$ a direct calculation gives~$
\least_6([L]_2)=a_1-a_2-a_3+a_6=-2$,
which is not divisible by~$6$.
Similarly for~$p=3$ we have~$\least_4([L]_3)=-a_2+a_4=-3+1=-2$.
which is not divisible by~$4$, so~$2,3\notin\gooddoldprimes(L)$.

Now assume that~$p\in\mathbb{P}\setminus\{2,3\}$.
We know that~$p\divides L_n$ if and only if~$n\in\lambda(p)(2\NN_0+1)$,
indeed we have the following `lifting the exponent' result for
the Lucas sequence:
\begin{equation*}\label{equationPAdicSizeLucasSequences}
\ord_p(L_{\lambda(p)k})
=
\ord_p(L_{\lambda(p)})+\ord_p(k)
\end{equation*}
for~$k\in2\NN_0+1$.
Thus by induction we have
\begin{equation}\label{equationPAdicSizeLucasSequences2}
[L_n]_p
=
\begin{cases}
p^{\ord_p(L_{\lambda(p)}+\ord_p(k))}&\mbox{if }n\in\lambda(p)k\mbox{ with $k$ odd};\\
1&\mbox{if not.}
\end{cases}
\end{equation}
Suppose now that~$p\in\gooddoldprimes(L)$ and
let~$c=\ord_p(a_r)$ so~$c>0$ and~$a_d=1$
for~$d<r$, for~$d\in2r\mathbb{N}$, and for~$d$ not a multiple of~$r$.

We first claim that every odd prime-power factor of~$2r$
divides~$p-1$.
To see this, write~$r=q^st$ with~$q\in\mathbb{P}\setminus\{2\}$
and~$q\notdivides t$ and let~$m=tp^j$ for some~$j\ge0$.
Then~$q\notdivides m$ since~$p\notdivides r$,
so~$mq^s=rp^j$ is an odd multiple of~$r$,
and hence~$a_{mq^s}=p^{c+j}$ by~\eqref{equationPAdicSizeLucasSequences2}.
By~\eqref{equationDoldForIntegerSequence2} we deduce
that~$p^{c_j}\equiv1$ modulo~$q^s$ and therefore in particular
\[
p^{c+1}-p^c=p^c(p-1)\equiv0\pmod{q^s}
\]
which shows that
\begin{equation}\label{equationcleveroldSawian}
q^s\divides p-1
\end{equation}
as claimed.

Now assume that~$r=2^st$ with~$s>0$ and~$t$ odd,
and write~$m=tp^j$ with~$j\ge0$. Since~$p\neq2$
we know that~$m$ is odd and so~$m2^s=rp^j$ is
once again an odd multiple of~$r$.
It follows that~$a_{m2^s}=p^{c+j}$.
On the other hand~$m2^{s+1}$ is an even multiple
of~$r$ and so~$a_{m2^{s+1}}=1$. By~\eqref{equationDoldForIntegerSequence2}
we deduce that~$p^{c+j}\equiv1$ modulo~$2^{s+1}$
and so in particular
\[
p^{c+1}-p^c=p^c(p-1)\equiv0\pmod{2^{s+1}}.
\]
Thus~$2^{s+1}\divides p-1$ which, combined
with~\eqref{equationcleveroldSawian} for every odd~$q$,
shows that~$2r\divides p-1$.

We now turn to the reverse direction.
If~$r$ does not exist (that is, if~$p$ does not divide
any term of~$L$) then~$p\in\gooddoldprimes(L)$
trivially. So suppose that~$r$ exists and~$2r\divides p-1$
and~$q\notdivides m$.
If~$m$ is not an odd multiple of~$r$ then
\[
a_{mp^s}=a_{mp^{s-1}}=1
\]
by~\eqref{equationPAdicSizeLucasSequences2}.
If~$m=rk$ for some odd~$k$ then~$p\notdivides k$ since~$p\notdivides m$,
and so
\[
a_{mp^s}-a_{mp^{s-1}}=p^{c+s}-p^{c+s-1}=p^{c+s-1}(p-1)\equiv0\pmod{p^s},
\]
which is compatible
with~\eqref{equationDoldForIntegerSequence2}.
If~$q\neq p$ and~$r\divides m$ then~$r\divides mq^s$ for any~$s\ge1$
and so
\[
\ord_{p}(mq^s/r)=\ord_p(m/r).
\]
It follows that~$a_{mq^s}\equiv q_{mq^{s-1}}$ modulo~$q^s$,
which again is compatible with~\eqref{equationDoldForIntegerSequence2}.
If~$r\notdivides m$ then the only way in which multiplication
by~$q$ can make~$r$ divide~$mq^s$ is for~$q$ to complete
the~$q$-part of~$r$. However,~$2r\divides p-1$ so~$p\equiv1$
modulo~$q^s$ for those~$s$ where this is possible.
It follows that~$a_{mq^s}\equiv a_{mq^{s-1}}$ modulo~$q^s$,
which completes the proof of~\eqref{equationDoldForIntegerSequence2}
and hence shows that~$p\in\gooddoldprimes(L)$.
\end{proof}

\begin{theorem}\label{theoremLucas}
The Lucas sequence~$L$ is almost Dold at every
prime, and
\[
\failure([L]_p)
=
\begin{cases}
3&\mbox{if }p=2,\\
\lambda(p)&\mbox{if }q\equiv3,7\pmod{20},\\
1&\mbox{in all other cases}.
\end{cases}
\]
\end{theorem}

\begin{proof}
Assume that~$p\neq2,3$.
We first prove that~$L$ is Dold at a prime
if and only if~$p\neq2$ and~$p$ is not congruent to~$3$
or~$7$ modulo~$20$, giving the third case.

If~$p\equiv\pm1$ modulo~$5$ then~$2\lambda(p)\divides p-1$,
and we may apply Lemma~\ref{lemmaMainLucas}
to see that~$p\in\gooddoldprimes(L)$.

If~$p\equiv\pm2$ modulo~$5$ then~$\varphi(p)\divides p+1$
and there are two cases to consider.
If~$p\equiv3$ modulo~$4$ then the Fibonacci rank~$\varphi(p)$
is even and hence~$p$ divides a Lucas number
and~$2\lambda(p)=\varphi(p)$.
The Dold condition requires that~$2\lambda(p)=\varphi(p)\divides p-1$
which is impossible as~$\varphi\divides p+1$ and~$\varphi(p)>2$
by inspection. Thus the simultaneous congruences~$p\equiv3$
modulo~$4$ and~$p\equiv\pm2$ modulo~$5$ imply that~$p\in\baddoldprimes(L)$.
By the Chinese remainder theorem these bad primes
are precisely those congruent to~$3$ or to~$7$ modulo~$20$.
If~$p\equiv1$ modulo~$4$ and~$p\equiv\pm2$ modulo~$5$
then Lemma~\ref{lemmaTrivialLucas} shows that~$p
\in\gooddoldprimes(L)$ trivially.

Applying the Chinese remainder theorem and checking
off the possible congruence classes modulo~$20$ then
proves that~$\failure([L]_p)=1$ unless~$p=2$ or~$p\equiv3,7$ modulo~$20$
and proves that~$\failure([L]_p)$ is not~$1$ (and
{\it{a priori}} may not exist) in all the other
cases.

The calculation of the claimed repair factors
for~$p=2$ and for~$p\equiv3,7$ modulo~$20$
involves precise knowledge of all the values
of~$\least_n([L]_p)$.
By Lengyel~\cite[Lemma~2]{MR1337793}
\[
[L_n]_2
=
\begin{cases}
1&\mbox{if }n\equiv1,2\pmod{3},\\
2&\mbox{if }n\equiv0\pmod{6},\\
4&\mbox{if }n\equiv3\pmod{6}
\end{cases}
\]
and so
\[
\least_n([L]_2)
=
\begin{cases}
1&\mbox{if }n=1,\\
3&\mbox{if }n=3,\\
-2&\mbox{if }n=6,\\
0&\mbox{for all other $n$}.
\end{cases}
\]
Thus~$\failure([L]_2)=3$ as claimed.

A direct calculation
using~\eqref{equationPAdicSizeLucasSequences2} shows
that
\begin{equation}\label{equationAA}
\least_{n}([L]_p)
=
\begin{cases}
\displaystyle\sum_{\substack{d\smalldivides m\\d\text{ odd}}}
\mu(\tfrac{m}{d})(p^{e+\ord_p(d)}-1)
&\mbox{if }n=rm,\\
0&\mbox{if }r\notdivides n.
\end{cases}
\end{equation}

If~$\lambda(p)$ is odd then~$[L_n]_p=1$ for all~$n\ge1$
and we are in the good prime case.
So assume that~$p\equiv3,7$ modulo~$20$
and~$\lambda(p)=r$.
We claim that if~$e=\ord_p(L_{r})$ the
only non-zero terms of~$\least_n([L]_p)$ are
\begin{align}
\least_1([L]_p)&=1,\nonumber\\
\least_{r}([L]_p)&=p^e-1\mbox{ if }r>1,\label{was12}\\
\least_{2r}([L]_p)&=1-p^e,\nonumber\\
\least_{rp^j}([L]_p)&=p^{e+j-1}(p-1)\mbox{ for }j\ge1,\nonumber\\
\least_{2rp^j}([L]_p)&=-p^{e+j-1}(p-1)\mbox{ for }j\ge1\label{was15}
\end{align}
by~\eqref{equationAA}.
This gives the remaining statement of the theorem
since it shows that~$p$
lies in~$\gooddoldprimes(L)$
if and only if~$\lambda(p)\divides p-1$ and that
\[
\failure([L]_p)
=
\frac{\lambda(p)}{\gcd(\lambda(p),p-1)}.
\]

If~$m=2^as$ with~$s$ odd then~\eqref{equationAA}
vanishes if~$a\ge2$ and for~$a=1$ is minus
its value at~$s$.
If~$a=0$ and~$r>1$ then~\eqref{equationAA} is~$p^e-1$
when~$s=1$ and for~$s>1$ we have
\begin{equation}\label{equationAB}
\least_{rs}([L]_p)
=
\sum_{d\smalldivides s}\mu(\tfrac{s}{d})a_d.
\end{equation}
The claimed equations~\eqref{was12}--\eqref{was15}
now follow from~\eqref{equationAA} and~\eqref{equationAB}.
\end{proof}

Turning to the Fibonacci numbers~$F=(F_n)=(1,1,2,\dots)$
we have a different starting point: The sequence~$F$ is not a
Dold sequence.
This fundamental distinction between~$L$ and~$F$ can be expressed
in two different languages. In dynamical systems~$L$ has a realization
as counting the periodic points of a simple dynamical
system (the `golden mean shift') while~$F$ does not
count the periodic points of any dynamical system.
In number theory~$L$ is a `trace sequence' because
\[
L_n
=
\trace A^n
\]
for all~$n\ge1$ where~$A=\begin{pmatrix}1&1\\1&0\end{pmatrix}$,
while~$F$ is not a trace sequence.

Indeed~$F$ is not even an almost
Dold sequence: Moss {\it{et al.}}~\cite[Lemma~1]{MR4394356} show that
the denominator of~$\frac{1}{p}\least_p(F)$ is~$p$ whenever~$p\equiv\pm2$ modulo~$5$.
In that work they show that~$(F_{n^k})$ is not almost
realizable if~$k$ is odd, but that~$(F_{5n^k})$ is
realizable if~$k$ is even. This was generalized by
Luca and one of the authors to certain other linear
recurrence sequences~\cite{MR4590321}.

\begin{theorem}\label{theoremFibonacci}
The Fibonacci sequence~$F$ is almost Dold at every prime,
and
\[
\failure([F]_p)
=
\begin{cases}
3&\mbox{if }p=2,\\
5&\mbox{if }p=5,\\
\frac{\varphi(p)}{\gcd(\varphi(p),p-1)}&\mbox{if $p\neq2,5$}.
\end{cases}
\]
\end{theorem}

\begin{proof}
The last claim in particular shows that~$p\in\gooddoldprimes(F)$ if and only
if~$p\equiv\pm1$ modulo~$5$, and in fact in all these cases~$p$
is known to be in~$\goodrealizableprimes$
by a construction due to Everest {\it{et al.}}~\cite[Th.~2.6]{MR1938222}.
The proof once again uses
well-known congruences and relations for the Fibonacci numbers,
for which we refer to Ribenboim's survey~\cite{MR1352481}.
For~$p\neq2,5$ we have~$\varphi(p)\divides p-(\tfrac{5}{p})$
and~$F_p\equiv(\frac{5}{p})$ modulo~$p$, so
in particular~$p\notdivides\varphi(p)$.
By Sanna~\cite{MR3512829} we have
\begin{equation}\label{fibonacci2part}
\ord_p(F_n)
=
\begin{cases}
\ord_5(n)&\mbox{if }p=5,\\
0&\mbox{if $p=2$ and $3\notdivides n$},\\
1&\mbox{if $p=2$ and $n\equiv3\pmod{6}$},\\
\ord_2(n)+2&\mbox{if $p=2$ and $6\divides n$}
\end{cases}
\end{equation}
for the anomalous primes~$2,5$ and
\begin{align*}
\ord_p(F_n)
=
\begin{cases}
e_p+\ord_p(n/\varphi(p))&\mbox{if }\varphi(p)\divides n,\\
0&\mbox{if }\varphi(p)\notdivides n
\end{cases}
\end{align*}
for all other primes.
Once again the argument uses these explicit formulas to
calculate that
\[
\least_{\varphi(p)}([F]_p)=p^{e_p}-1
\]
and
\[
\least_{p^k\varphi(p)}([F]_p)=p^{e_p+k-1}(p-1)
\]
for all~$k\ge1$, with all other values vanishing. It follows that
\[
\failure([F]_p)
=
\frac{\varphi(p)}{\gcd(\varphi(p),p-1)}.
\]
This (after a somewhat tedious calculation)
completes the theorem for~$p\neq2,5$.

For~$p=5$ a similar calculation using~\eqref{fibonacci2part}
shows that
\begin{align*}
\least_1([F]_5)&=1,\\
\least_{5^k}([F]_5)&=4\cdot 5^{k-1}\mbox{ for }k\ge1
\end{align*}
with all other values vanishing
and so~$\failure([F]_5)=5$.
Finally, for~$p=2$ a calculation using~\eqref{fibonacci2part}
again shows that
\begin{align*}
\least_1([F]_2)&=1,\\
\least_3([F]_2)&=1,\\
\least_6([F]_2)&=6,\\
\least_{3\cdot2^k}([F]_2)&=2^{k+1}\mbox{ for }k\ge2\\
\end{align*}
with all other values vanishing, and so~$\failure([F]_2)=3$.
\end{proof}

\section{Relative Densities of the Good Primes}

The relative natural density of a
set~$\mathcal{S}\subseteq\PP$ is defined
to be
\begin{equation*}\label{definitionRelativePrimeDensity}
\delta_{\PP}(\mathcal{S})
=
\lim_{x\to\infty}
\frac{\bigl\vert\{p\le x:p\in\mathcal{S}\}\bigr\vert}{\bigl\vert\{p\le x:p\in\PP\}\bigr\vert}
\end{equation*}
when this limit exists. Put~$
\mathcal{P}_L
=
\{p\in\PP:p\divides L_n\text{ for some }n\ge1\}$.
We first recall the set of good primes for each sequence,
since the Dold property and realizability do not coincide locally for Lucas.

\begin{proposition}\label{propositionGoodPrimeSets}
For the Lucas numbers,
\begin{align}
\gooddoldprimes(L)
&=
\PP\setminus\bigl\{\{2\}\cup
\{p\in\PP:p\equiv3\text{ or }7\pmod{20}\}\bigr\},
\label{equationLucasDoldSet}\\
\gooddoldprimes^{\text{almost}}(L)
&=\PP,
\nonumber
\intertext{and}
\goodrealizableprimes(L)
&=
\goodrealizableprimes^{\text{almost}}(L)
=
\PP\setminus\mathcal{P}_L
=
\{p\in\PP:\lambda(p)=\infty\}.
\label{equationLucasRealizableSets}
\end{align}
For the Fibonacci numbers,
\begin{align}
\gooddoldprimes(F)
&=
\goodrealizableprimes(F)
=
\{p\in\PP:p\equiv1\text{ or }4\pmod5\},
\label{equationFibonacciGoodSets}
\intertext{and}
\gooddoldprimes^{\text{almost}}(F)
&=
\goodrealizableprimes^{\text{almost}}(F)
=\PP.
\label{equationFibonacciAlmostSets}
\end{align}
\end{proposition}

\begin{proof}
The assertions about Dold and almost Dold primes
for the Lucas sequence follow from
Theorem~\ref{theoremLucas}. It remains to identify the realizable sets.

If~$\lambda(p)=\infty$, then~$[L]_p=(1,1,\ldots)$, which is realizable.
Conversely, let~$p$ be odd and put~$r=\lambda(p)<\infty$
and~$e=\ord_p(L_r)$. The indices of Lucas numbers divisible by~$p$ are
the odd multiples of~$r$. Thus, among the divisors of~$2r$, only~$r$
indexes a term divisible by~$p$. Since~$2r>1$, it follows that
\[
\least_{2r}([L]_p)
=
\sum_{d\smalldivides 2r}\mu(2r/d)+\mu(2)(p^e-1)
=1-p^e<0.
\]
For~$p=2$, the proof of Theorem~\ref{theoremLucas}
shows that~$\least_6([L]_2)=-2$. These negative values exclude both
realizability and almost realizability.
It follows that~$[L]_p$ is realizable if
and only if it is almost realizable, and this occurs exactly
when~$p\notin\mathcal{P}_L$.

For Fibonacci numbers, Theorem~\ref{theoremFibonacci} gives
\[
\failure([F]_p)
=
\frac{\varphi(p)}{\gcd(\varphi(p),p-1)}
\qquad(p\neq2,5).
\]
If~$\bigl(\frac5p\bigr)=1$,
then~$\varphi(p)\divides p-1$ and this factor is~$1$.
If~$\bigl(\frac5p\bigr)=-1$ and the factor were~$1$,
then~$\varphi(p)$ would divide both~$p-1$ and~$p+1$, and hence would
divide~$2$, contradicting~$\varphi(p)\ge3$. Together with the
exceptional factors at~$2$ and~$5$, the
identity~$\bigl(\frac5p\bigr)=\bigl(\frac p5\bigr)$ from quadratic reciprocity
gives~\eqref{equationFibonacciGoodSets}.
Finally,~$\least_1([F]_p)=1$ for every~$p$, and every other nonzero value
displayed in the proof of Theorem~\ref{theoremFibonacci} is
non-negative. Thus the Dold property and realizability coincide, while every
local Fibonacci sequence is almost realizable,
proving~\eqref{equationFibonacciAlmostSets}.
\end{proof}

\begin{corollary}[Relative prime densities]\label{corollaryDensities}
All eight relative natural densities below exist and are given by
\begin{equation}\label{equationDensityTable}
\begin{array}{c@{\qquad}cccc}
\toprule
&\gooddoldprimes
&\goodrealizableprimes
&\gooddoldprimes^{\text{almost}}
&\goodrealizableprimes^{\text{almost}}\\
\midrule
L&\frac34&\frac13&1&\frac13\\[4pt]
F&\frac12&\frac12&1&1\\
\bottomrule
\end{array}
\end{equation}
Moreover, the set of primes that are good
for the Dold property and primes that are locally
realizable have densities
\begin{align*}
\delta_{\PP}\!\left(
\gooddoldprimes(L)\setminus\goodrealizableprimes(L)\right)
&=\textstyle\frac5{12}
\intertext{and}
\delta_{\PP}\!\left(
\gooddoldprimes^{\text{almost}}(L)
\setminus\goodrealizableprimes^{\text{almost}}(L)\right)
&=\textstyle\frac23.
\end{align*}
The corresponding Fibonacci gaps are empty.
\end{corollary}

\begin{proof}
By the prime number theorem for arithmetic progressions,
each of the eight reduced residue classes
modulo~$20$ has relative prime density~$\frac18$. Apart from the single
prime~$2$, equation~\eqref{equationLucasDoldSet} excludes exactly the
classes~$3$ and~$7$. Hence
\[
\delta_{\PP}(\gooddoldprimes(L))=\textstyle\frac34,
\qquad
\delta_{\PP}(\gooddoldprimes^{\text{almost}}(L))=1.
\]

Lagarias proved that~$\delta_{\PP}(\mathcal{P}_L)=\frac23$
in~\cite[Th.~B]{Lagarias1985}. His indexing includes~$L_0=2$, whereas
ours starts at~$L_1$; this does not change the prime-divisor set
because~$2\divides L_3$
(an erratum~\cite{Lagarias1994} concerns a different recurrence and leaves
Theorem~B unchanged). Thus~\eqref{equationLucasRealizableSets}
gives density~$\frac13$ for both Lucas
realizability sets.

For Fibonacci numbers, the classes~$1$ and~$4$ modulo~$5$ each have
relative prime density~$\frac14$.
Equations~\eqref{equationFibonacciGoodSets}--\eqref{equationFibonacciAlmostSets}
give the four entries in the Fibonacci row of~\eqref{equationDensityTable}.

Finally, realizability implies the Dold congruences, so
$\goodrealizableprimes(L)\subseteq\gooddoldprimes(L)$. The first gap
therefore has density~$\frac34-\frac13=\frac{5}{12}$.
The second gap is exactly~$\mathcal{P}_L$, and hence has density~$\frac23$.
\end{proof}

\section{Remarks and open problems}

The distinction between~$\gooddoldprimes$ and
$\goodrealizableprimes$ is essential for the Lucas sequence:
Indeed the negative transform at~$2\lambda(p)$ produces the positive-density
gaps in Corollary~\ref{corollaryDensities}. By contrast, all the local
Fibonacci transforms are non-negative, so the Dold and realizability
sets coincide in both the exact and almost settings.

We have focused only on the Fibonacci sequence and the classical
Lucas sequence because these two cases are good representatives of
the phenomena of interest and have convenient
known arithmetic properties. The Lucas sequence is a trace sequence and hence
intimately linked to dynamical systems in a natural way; the
Fibonacci sequence is not a trace
sequence but nonetheless has tractable arithmetic properties for different
reasons. It is natural to ask how far the congruence and sign
arguments presented here extend to general companion Lucas sequences~$V(P,Q)$ and
Lucas sequences~$U(P,Q)$ in the standard notation of
Ribenboim's survey~\cite{MR1352481}.

Another natural problem from the results here is to try and
characterize these sets of primes for linear recurrence sequences
in general,
starting with the intriguing observation that the classical
tribonacci sequence~$(1,1,2,4,\dots)$
and~$4$-bonacci sequence~$(1,1,1,2,5,\dots)$ seem
(computationally) to have~$\gooddoldprimes=\varnothing$.

The notion of~$\gooddoldprimes(a)$ and~$\goodrealizableprimes(a)$
makes sense for any integer
sequence~$a$. The implicit problem of characterizing these
sets for a given sequence is already raised in~\cite{MR5076929}
with sample empirical calculations for some graph-counting
and related sequences.
For Lehmer--Pierce sequences (that is, those with~$n$th
term of the form
\[
a_n=\prod_{j=1}^{d}\vert\lambda_j^n-1\vert
\]
where~$\lambda_1,\dots,\lambda_j$
are the eigenvalues of an element of~${\rm{SL}}_d(\ZZ)$ with no
unit root as an eigenvalue) the local algebraic
realizability results from~\cite{MR5076929}
show that~$\goodrealizableprimes(a)=\PP$.
More generally, the same holds for the sequence of
periodic point counts from any of the so-called $S$-integer
dynamical systems studied by Everest {\it{et al.}}~\cite{MR1461206,MR2339472}.

\section*{Acknowledgements}
We dedicate this paper to the memory of our
inspirational and generous friend Florian Luca
who took such an interest in questions of this
sort.

This work was supported by the KKU Outbound
Visiting Scholar programme of Khon Kaen University.
The fourth named author would like to thank
Jakub Byszewski for some useful comments on the
empirical results of~\cite{MR5076929}.

Many of the initial numerical experiments that
motivated this work were run using PARI/GP~\cite{PARI2},
and many of the algebraic calculations and
later numerical checks were assisted by
ChatGPT~\cite{ChatGPT}. None of the writing was
assisted by AI.

\section*{Disclosure statement}
No conflict of interest has been reported by the authors.



\noindent MSC2020: 11B39, 37P35


\end{document}